\documentclass{amsart}\usepackage{ifthen}\newboolean{amsart}\setboolean{amsart}{true}

\usepackage{comment}

\newcommand{\fdim}{\dim^*}
\renewcommand{\v}{\mathbf}

\newcommand{\tpt}{\mathbf{2}+\mathbf{2}}

\theoremstyle{plain}
\newtheorem{theorem}{Theorem}
\newtheorem{lemma}[theorem]{Lemma}

\newtheorem{proposition}[theorem]{Proposition}

\theoremstyle{definition}
\newtheorem{definition}[theorem]{Definition}

\newtheorem{conjecture}[theorem]{Conjecture}

\theoremstyle{remark}

\title{The proof of the removable pair conjecture for fractional dimension}

\ifthenelse{\boolean{amsart}}
{
\author{Csaba Bir\'o}
\address{Department of Mathematics, University of Louisville, Louisville KY
40292, USA} \email{csaba.biro@louisville.edu}
 
\author{Peter Hamburger}
\address{Department of Mathematics, Western Kentucky
University, Bowling Green, KY 42101, USA}
\email{peter.hamburger@wku.edu}
 
\author{Attila P\'or}
\address{Department of Mathematics, Western Kentucky
University, Bowling Green, KY 42101, USA} \email{attila.por@wku.edu}
}
{
\author{Csaba Bir\'o\\
\small Department of Mathematics\\[-0.8ex]
\small University of Louisville\\[-0.8ex]
\small Louisville, KY 40292, U.S.A.\\
\small\tt csaba.biro@louisville.edu\\
\and
Peter Hamburger\\
\small Department of Mathematics\\[-0.8ex]
\small Western Kentucky University\\[-0.8ex]
\small Bowling Green, KY 42101, U.S.A.\\
\small\tt peter.hamburger@wku.edu\\
\and
Attila P\'or\\
\small Department of Mathematics\\[-0.8ex]
\small Western Kentucky University\\[-0.8ex]
\small Bowling Green, KY 42101, U.S.A.\\
\small\tt attila.por@wku.edu\\}

\date{\dateline{Jan 4, 2014}{not yet}}
}

\begin{document}

\ifthenelse{\boolean{amsart}}{}{\maketitle}

\begin{abstract}
In 1971 Trotter (or Bogart and Trotter) conjectured that every finite poset on
at least $3$ points has a pair whose removal does not decrease the dimension by
more than $1$. In 1992 Brightwell and Scheinerman introduced fractional
dimension of posets, and they made a similar conjecture for fractional
dimension. This paper settles this latter conjecture.
\end{abstract}

\ifthenelse{\boolean{amsart}}{\maketitle}{}

\section{Introduction}

\subsection{Dimension of posets}
Let $P$ be a poset. A set of its linear extensions $\{L_1,\ldots,L_d\}$ forms a
\emph{realizer}, if $L_1\cap\cdots\cap L_d=P$. The minimum cardinality of a
realizer is called the \emph{dimension} of the poset $P$, denoted by $\dim(P)$. This concept is also
sometimes called the \emph{order dimension} or the \emph{Dushnik-Miller
dimension of the partial order} as it was introduced in \cite{Dus-Mil-41}.

The dimension is ``continuous'' in the sense that the removal of a point can
never decrease the dimension by more than $1$.  If the removal of a pair of
points decreases the dimension by at most $1$, such a pair is called a
\emph{removable pair}. The following conjecture has become known as the
\emph{Removable Pair Conjecture}.
\begin{conjecture}[\cite{Tro-CAPOS}, pp.\ 26]
Every poset on at least $3$ point has a removable pair.
\end{conjecture}

The origins of the conjecture are not entirely clear. It appeared in print in a
1975 paper by Trotter \cite{Tro-75}, but according to Trotter \cite{Trotter}, it
was probably formulated at the 1971 summer conference on combinatorics held at
Bowdoin College, and should be credited either to Trotter or to Bogart and
Trotter.

Let $P$ denote a poset. In the following, $x\|y$ denotes that $x$ is
incomparable to $y$ in $P$. We will also use the following standard notations.
\begin{gather*}
D(x)=\{y\in P: y<x\}\qquad D[x]=D(x)\cup\{x\}\\
U(x)=\{y\in P: y>x\}\qquad U[x]=U(x)\cup\{x\}\\
I(x)=\{y\in P: y\|x\}\\
\min(P)=\{x\in P:D(x)=\emptyset\}\qquad\max(P)=\{x\in P: U(x)=\emptyset\}
\end{gather*}
\begin{definition}
An ordered pair of vertices $(x,y)$ is called a \emph{critical pair}, if
$x\|y$, $D(x)\subseteq D(y)$, and $U(y)\subseteq U(x)$. A linear extension $L$
\emph{reverses} the critical pair $(x,y)$ if $y<x$ in $L$. A set of linear
extensions \emph{reverses} a critical pair, if the pair is reversed in at least
one of the linear extensions.
\end{definition}
The following proposition expresses that the critical pairs are the only
significant incomparable pairs for constructing realizers.
\begin{proposition}[\cite{Rab-Riv-79}]
A set of linear extensions is a realizer if and only if it reverses every critical
pair.
\end{proposition}

\subsection{Fractional dimension}
Determining the dimension of a poset can be regarded as a linear integer
programming problem. Let $P$ be a poset and $\{L_1,\ldots,L_\ell\}$ the set of
its linear extensions and $\{(a_1,b_1),\ldots,(a_c,b_c)\}$ the set of its
critical pairs. Let $A=[a_{ij}]$ be a $c\times \ell$ binary matrix, where
$a_{ij}=1$ iff $(a_i,b_i)$ is reversed in $L_\ell$. The following integer
program gives the dimension of $P$.
\begin{gather*}
A\v{x}\geq\v{1}\\
\v{x}\geq\v{0}\\
\v{x}\in\mathbb{Z}^\ell\\
\text{min }\v{1}^T\v{x}
\end{gather*}

In 1992 Brightwell and Scheinerman \cite{Bri-Sch-92} introduced the notion of
fractional dimension of posets as the optimal solution of the linear relaxation
of this integer program. They used the notation $\mathrm{fdim}(P)$ for
fractional dimension, but to keep the notation consistent with fractional graph
theory, we will use $\fdim(P)$. A feasible solution of the linear program will
be called an \emph{f-realizer}.

If we consider the $\ell$-dimensional vector space generated by the abstract
basis $L_1,\ldots,L_\ell$, then an
f-realizer is a linear combination $\sum\alpha_i L_i$ with
$0\leq\alpha_i\leq 1$, where for all
critical pairs $(x,y)$, we have $\sum_{i:y<x\text{ in }L_i}\alpha_i\geq 1$.
In fact, for a linear combination $\sum\alpha_i L_i$, we will say that it
reverses the critical pair $(x,y)$ $\alpha$ times, if
$\sum_{i:y<x\text{ in }L_i}\alpha_i=\alpha$. If a critical pair is reversed at
least once ($1$ times), we will simply say it is reversed. This way, a linear
combination is an f-realizer if and only if it reverses every critical pair. The
\emph{weight} of an f-realizer is $\sum_{i=1}^\ell\alpha_i$, and the fractional
dimension of $P$ is the minimum weight of an f-realizer.

Clearly, for all posets $P$, $\fdim(P)\leq\dim(P)$, but as shown by Brightwell
and Scheinerman \cite{Bri-Sch-92} (with all the other results in this
paragraph), the ordinary dimension and the fractional dimension can be
arbitrarily far apart.  Nevertheless, there exist posets with arbitrarily large
fractional dimension.  The continuous property translates exactly for
fractional dimension: for any element $x\in P$, $\fdim(P-x)\geq \fdim(P)-1$.

In their paper Brightwell and Scheinerman conjectured that the fractional
version of the Removable Pair Conjecture holds: there is always a pair of
points that decreases the fractional dimension by at most one. In 1994, Felsner
and Trotter \cite{Fel-Tro-94} suggested a weakening of the question: is there
an absolute constant $\varepsilon>0$ so that any poset with $3$ or more points
always contains a pair whose removal decreases the fractional dimension by at
most $2-\varepsilon$? In this paper we prove the Brightwell--Scheinerman 
conjecture, which is of course equivalent to the Felsner--Trotter conjecture
with $\varepsilon=1$.

We would like to note here that the original definition of fractional dimension
by Brightwell and Scheinerman is different from what we introduced above. For
completeness, let us give their original, equivalent definition here. A
\emph{$k$-fold realizer} is a multiset of linear extensions such that for any
incomparable pair $x,y$ there exists at least $k$ linear extensions (with
multiplicity) in the multiset in which $x<y$, and there exists at
least $k$ other linear extensions in which $x>y$. Let $t(k)$ be the minimum
cardinality of a $k$-fold realizer. The fractional dimension of the poset is
$\lim_{k\to\infty}t(k)/k=\inf t(k)/k$.

\subsection{Interval orders}

Let $P$ be a poset, and suppose that there is a map $f$ from $P$ to the set of
closed intervals of the real line, so that $x < y$ in $P$ if and only if the
right endpoint of $f(x)$ is less than (in the real number system) the left endpoint of
$f (y)$. We say that the multiset of intervals $\{f(x):x\in P\}$ is an \emph{interval
representation} of $P$. If a poset has an interval representation, it is an
\emph{interval order}.

The poset $\tpt$ denotes the poset with ground set $\{a_1,b_1,a_2,b_2\}$, where
the only relations are $a_1<b_1$ and $a_2<b_2$. If a poset contains
$\tpt$ it can not be an interval order. In fact this property characterizes
interval orders (see Fishburn \cite{Fis-70}).

\begin{theorem}[\cite{Fis-70}]\label{thm:charint}
A poset is an interval order if and only if it does not contain $\tpt$.
\end{theorem}

An appealing property of interval orders is that they admit a
positive answer to the Removable Pair Conjecture.

\begin{theorem}[\cite{Tro-97}]\label{thm:rpio}
Any interval order on at least $3$ points contains a removable pair.
\end{theorem}

\section{The existence of a removable pair}

The following lemma is the fractional analogue of Theorem~\ref{thm:rpio}.

\begin{lemma}\label{lemma:interval}
Let $P$ be an interval order on at least $3$ points. Then there exist points
$a,b\in P$ such that $\fdim(P-a-b)\geq\fdim(P)-1$.
\end{lemma}

\begin{proof}
If $P$ is an antichain, then $\fdim(P)=\dim(P)=2$, so the statement follows.
Otherwise let $a<b$ two elements of $P$ such that $a\in\min(P)$, $b\in\max(P)$,
and $I(a)\subseteq\min(P)$, $I(b)\subseteq\max(P)$. Such elements always exist:
consider a representation, and choose $a$ to be an interval whose right
endpoint is the leftmost, and $b$ to be an interval whose left endpoint is the
rightmost. The fact that $P$ is not an antichain ensures $a<b$.

Let $I=I(a)\cap I(b)$, and $R=P\setminus(\{a,b\}\cup I(a)\cup I(b))$. Note that
$I$ consists of isolated elements.  Let $Q=P-\{a,b\}$, and let $\sum\alpha_i
L_i$ be an f-realizer of $Q$ of weight $\fdim(Q)$. For each $i$, we will modify
$L_i$ to get a linear extension $L_i'$ of $P$ as follows. Let $L_i$ be such
that in $L_i$ we have $I<a<L<b$, where $I$ is ordered the same way for each
$i$, and $L$ is an ordered set of elements of $P\setminus(\{a,b\}\cup I)$, such that
the ordering preserves that of $L_i$.  We construct one additional linear
extension $L$ so that $I(a)\setminus I<a<R<b<I(b)\setminus I<I$, and the
ordering of $I$ is the reverse of that of the $L_i$'s.

We claim that $\sum\alpha_i L_i'+L$ is an f-realizer of $P$. Critical pairs not
involving elements of $I\cup\{a,b\}$ are obviously reversed, and the rest are
reversed because $\sum \alpha_i\geq 1$. This shows $\fdim(P)\leq\fdim(Q)+1$.
\end{proof}

Note that the lemma above can be proven using more of existing machinery.
Proceeding by induction on $|P|$, it is sufficient to prove that the statement
holds in case $P$ is not an antichain and $P$ is indecomposable with respect to
lexicographic sum. Using this, we may assume that in the above argument $I=\emptyset$,
and then $L$ may be defined by $I(a) < a < R < b < I(b)$.

Also note that Trotter \cite{Tro-97} gave a proof of the Removable Pair Conjecture for
interval orders, and that proof can be translated to work with fractional
dimension.

\begin{theorem}
Let $P$ be a poset with at least $3$ elements. There exists a pair whose
removal decreases the fractional dimension by at most $1$.
\end{theorem}
\begin{proof}
By Lemma~\ref{lemma:interval} we may assume that $P$ is not an interval order,
so by Theorem~\ref{thm:charint} $P$ has elements $a_1$, $b_1$, $a_2$, and
$b_2$, such that $a_1<b_1$, $a_2<b_2$, and no other comparabilities between any
two of these.

Let $P_1=P-\{a_1,b_2\}$, $P_2=P-\{a_2,b_1\}$. Let $\sum\alpha_i L_i$, and
$\sum\beta_i M_i$ f-realizers of $P_1$ and $P_2$ of weights $\fdim(P_1)$ and
$\fdim(P_2)$, respectively. We extend the linear extensions $L_i$ and $M_i$ for
each $i$ by inserting the missing elements $a_1,b_2$ or $a_2,b_1$ at arbitrary
valid positions to get $L_i'$ and $M_i'$, respectively.

Define $D^-(b_1)=D(b_1)\setminus D[a_1]$, and $D^-(b_2)=D(b_2)\setminus D[a_2]$,
and similarly, $U^-(a_1)=U(a_1)\setminus U[b_1]$, and $U^-(a_2)=U(a_2)\setminus
U[b_2]$.
Let
\begin{align*}
%a_1,b_1 lehuzva, a_2,b_2 feltolva ameddig csak lehet:
L_1 &= D(a_1) < a_1 < D^-(b_1) < b_1 < R_1 < a_2 < U^-(a_2) < b_2 < U(b_2)\\
%a_2,b_2 lehuzva, a_1,b_1 feltolva ameddig csak lehet:
L_2 &= D(a_2) < a_2 < D^-(b_2) < b_2 < R_2 < a_1 < U^-(a_1) < b_1 < U(b_1)
\end{align*}
linear extensions, where $R_1$ and $R_2$ represent the rest of the elements as
appropriate. These are exactly the linear extensions that appear in an article
by Bogart \cite{Bog-73}, where he proves their existence in a more general
setting.
We claim that
\begin{equation}\label{eq:lincomb}
\frac{1}{2}\sum\alpha_i L_i'+\frac{1}{2}\sum\beta_i
M_i'+\frac{1}{2}L_1+\frac{1}{2}L_2
\end{equation}
is an f-realizer of $P$.

Indeed, every critical pair $(x,y)$ is reversed in at least two terms of
(\ref{eq:lincomb}). If neither of $x,y$ is in $\{ a_1, a_2, b_1, b_2 \}$ then
it is reversed in the first and the second term. The critical pair $(a_1,b_2)$
is reversed in the second and fourth term, and the pair $(a_2,b_1)$ is reversed
in the first and third term.

It remains to be seen that $(x,y)$ gets reversed if exactly one of $x$ and $y$
is in the set $\{a_1,a_2,b_1,b_2\}$. Up to symmetry, there are four such critical
pairs $(x,a_1)$, $(a_1,y)$, $(x,b_2)$, $(b_2,y)$. All of them get reversed in
the second term, and they get reversed in $L_1,L_2,L_2,L_1$, respectively.

We have shown that
\[
\frac{1}{2}\fdim(P_1)+\frac{1}{2}\fdim(P_2)+1\geq\fdim(P),
\]
which implies that at least one of $P_1$ or $P_2$ has fractional dimension at
least $\fdim(P)-1$.
\end{proof}

\section{Acknowledgment}
The authors would like to thank the referees for valuable suggestions on
improvements of the paper. 

\ifthenelse{\boolean{amsart}}{
\bibliographystyle{amsplain}
\bibliography{bib,extra}
}
{
\providecommand{\bysame}{\leavevmode\hbox to3em{\hrulefill}\thinspace}
\providecommand{\MR}{\relax\ifhmode\unskip\space\fi MR }
\providecommand{\MRhref}[2]{%
  \href{http://www.ams.org/mathscinet-getitem?mr=#1}{#2}
}
\providecommand{\href}[2]{#2}

}

\end{document}